\documentclass[a4paper,reqno,nosumlimits]{amsart}
\usepackage{amsmath}
\usepackage[all]{xy}
\usepackage{amssymb}
\usepackage{amsfonts}
\usepackage[mathscr]{eucal}
\usepackage{enumerate,color}
\usepackage{graphicx}

\theoremstyle{definition}
\newtheorem{definition}{Definition}[section]
\theoremstyle{remark}

\newtheoremstyle{estilo}{0mm}{0mm}{\slshape}{}{\bfseries}{.}{ }{}
\theoremstyle{theorem}
\newtheorem{theorem}[definition]{Theorem}
\newtheorem{corollary}[definition]{Corollary}
\newtheorem{lemma}[definition]{Lemma}

\newtheoremstyle{mystyle}{2mm}{0mm}{}{}{\bfseries}{}{ }{\thmnumber{#2}.\thmnote{ #3}}
\theoremstyle{mystyle}
\newtheorem{fact}[definition]{}
\newtheoremstyle{myremark}{2mm}{0mm}{}{}{\bfseries}{}{ }{\thmname{#1} \thmnumber{#2}. \thmnote{ #3}}
\theoremstyle{myremark}

\DeclareMathOperator{\Id}{Id}
\DeclareMathOperator{\Ch}{Ch}
\DeclareMathOperator{\End}{End}
\DeclareMathOperator{\Hom}{Hom}

\DeclareMathOperator{\Mod}{Mod}

\DeclareMathOperator{\Tr}{tr}\DeclareMathOperator{\ev}{ev}

\email{}
\thanks{}

\input{tcilatex}

\begin{document}
\title{On De Rham Cohomology of Linear Categories}
\author{Andrei Chite\c{s}, M\u{a}d\u{a}lin Ciungu and Drago\c s \c Stefan}
\address{University of Bucharest, Faculty of Mathematics and Computer
Science, 14 Academiei Street, Bucharest Ro-010014, Romania}
\email{andrei.chites@pointlogistix.ro}
\email{madalinciungu@yahoo.com}
\email{dragos.stefan@fmi.unibuc.ro}

\begin{abstract}
We define the Chern map from the Grothendieck group of a linear category $%
\mathcal{C}$ to the de Rham cohomology of $\mathcal{C}$ with coefficients in
a $DG$-category $\Omega^{\ast}$. In order to achieve our goal, we define the
notion of connection on a $\mathcal{C}$-module, and we show that the trace
of the curvature of a connection is a de Rham cocycle, whose cohomology
class does not depend on the choice of the connection.
\end{abstract}

\maketitle

\section*{Introduction}

The classical Chern character is a morphism from the $K_{0}$-theory group of
a manifold to its de Rham cohomology groups. In his work on the foliations
of a manifold, Alain Connes extended the construction of the Chern character
to show that there is a pairing between the Grothendieck group $K_{0}(A)$ of
the Banach algebra $A$ associated to a foliation and the cyclic cohomology $%
A $, which replaces de Rham cohomology in Noncommutative Geometry ; see \cite%
{Co}. In \cite{Ka}, the definition of the Chern map was extended to $%
K_{n}(A),$ the higher $K$-theory groups of $A.$ The method used in loc. cit.
is purely algebraic, working for any finitely generated projective module
over a, not necessarily commutative, algebra $A$. See also \cite[Chapter 8]%
{Lo}.

Linear categories were defined and investigated in the seminal paper \cite%
{Mi}. Since their appearance, they have been used as a very important tool
not only in algebra, but in many other fields, including algebraic topology,
logic, computer science, etc.

In this paper, proceeding as in \cite{Ka}, we define a Chern map from the
Grothendieck group $K_{0}(\mathcal{C})$ of a linear category $\mathcal{C}$
to the de Rham cohomology of $\mathcal{C}.$ For, we recall the definition of
linear categories and their basic properties in the first section of the
paper. In this part, we also recall some homological properties of the
category of (right) modules over a linear category, and show that the
Hattori-Stallings trace map can be defined for any finitely generated
projective module over a linear category.

In the second section we deal with connections on a finitely generated right 
$\mathcal{C}$-module $M$. As an example, we introduce the Levi-Civita
connection on a finitely generated projective module. We are also able to
describe all connections on a finitely generated free $\mathcal{C}$-module.
In the last part of this section, we define the curvature of a connection
and we compute it for the above mentioned examples of connections.

In the third section we define the de Rham cohomology of a linear category $%
\mathcal{C}$, with respect to a $DG$-category $\Omega ^{\ast }$ such that $%
\Omega ^{0}=\mathcal{C}.$ We prove that the trace of the curvature of a
connection $\nabla$ on a finitely generated projective $\mathcal{C}$-module $%
M$ is a de Rham $2q$-cocycle $\Ch^{q}(M,\nabla ).$ We also show that the de
Rham cohomology class of $\Ch^{q}(M,\nabla )$ does not depend on the choice
of the connection $\nabla $ on $M.$ We shall denote this cohomology class by 
$\Ch^{q}(M)$. In particular we deduce our main result, stating that $%
[M]\mapsto \Ch^q(M)$ defines a morphism from the Grothendieck group $K_{0}(%
\mathcal{C})$ of $\mathcal{C}$ to the de Rham cohomology of $\mathcal{C}$
with coefficients in $\Omega ^{\ast }.$

\section{Preliminaries}

In this section we recall the definitions of the notions that we work with.
Throughout, $\Bbbk $ will denote a commutative field. The tensor product
over $\Bbbk $ will be denoted by $\otimes .$

\begin{fact}[$\Bbbk $-linear categories.]
A category $\mathcal{C}$\ is called $\Bbbk $-linear if every hom-set in $%
\mathcal{C}$\ is a $\Bbbk $-linear space and the composition maps in $%
\mathcal{C}$\ are bilinear. We shall denote the space of morphisms from $x$
to $y$ by $_{y}\mathcal{C}_{x}$, so the composition map can be seen as a
linear transformation from $_{x}\mathcal{C}_{y}\otimes {}_{y}\mathcal{C}%
_{z}\ $to ${}_{x}\mathcal{C}_{z},$ for any objects $x$, $y$ and $z$ in $%
\mathcal{C}$. In this paper $\mathcal{C}$ will always denote a small linear
category. Thus $\mathcal{C}_{0},$ the class of objects in $\mathcal{C},$ is
a set.

A functor $F:\mathcal{C}\longrightarrow \mathcal{D}$ between linear
categories is $\Bbbk $-linear if the corresponding map from ${}_{x}\mathcal{C%
}_{y}$ to ${}_{F(x)}\mathcal{D}_{F(y)}$ is $\Bbbk $-linear, for any objects $%
x$ and $y$ in $\mathcal{C}_{0}.$
\end{fact}

\begin{fact}[Examples of $\Bbbk $-linear categories.]
Of course the category of $\Bbbk $-linear spaces is the prototype of $\Bbbk $%
-linear categories. Other examples are listed below.

\begin{enumerate}
\item Let $\mathcal{C}$\ be a given $\Bbbk $-linear category. The opposite
category of $\mathcal{C}$ is also a linear category. Recall that $\mathcal{C}
$\ and $\mathcal{C}^{op}$\ have the same objects, while $_{x}(\mathcal{C}%
^{op})_{y}={}_{y}\mathcal{C}_{x}$. If we denote the composition in $\mathcal{%
C}^{op}$ by $\bullet $, then $f\bullet g=g\circ f$\ \ for any $g\in {}_{x}%
\mathcal{C}_{y}$\ and $f\in {}_{y}\mathcal{C}_{z}$.

\item Let $\mathcal{C}$\ and $\mathcal{D}$\ be two $\Bbbk $-linear
categories. The tensor product $\mathcal{C\boxtimes D}$\ is the $\Bbbk $%
-linear category:

\begin{enumerate}
\item With the set of objects $\mathcal{(C\boxtimes D)}_{0}=\mathcal{C}_{0}%
\mathcal{\times D}_{0}$;

\item The space of morphisms $_{(x,y)}\mathcal{(C\boxtimes D)}_{(x^{\prime
},y^{\prime })}={}_{x}\mathcal{C}_{x^{\prime }}\otimes {}_{y}\mathcal{D}%
_{y^{\prime }}$, for any $(x,y)$ and $(x^{\prime },y^{\prime })$ in $%
\mathcal{C}_{0}\times \mathcal{D}_{0}$.
\end{enumerate}

The composition in $\mathcal{C\boxtimes D}$\ is defined by%
\begin{equation*}
(f\otimes g)\circ \ (f^{\prime }\otimes g^{\prime })=(f\circ f^{\prime
})\otimes (g\circ g^{\prime }),
\end{equation*}%
for any $f\otimes g\in {}_{x}\mathcal{C}_{y}\otimes {}_{x^{\prime }}\mathcal{%
D}_{y^{\prime }},$ and $f^{\prime }\otimes g^{\prime }\in {}_{y}\mathcal{C}%
_{z}\otimes {y^{\prime }}\mathcal{D}_{z^{\prime }}$. In particular, for a
linear category $\mathcal{C}$, one defines its enveloping category by $%
\mathcal{C}^{e}:=\mathcal{C\boxtimes C}^{op}$.

\item The category of left $\mathcal{C}$-modules is a $\Bbbk $-linear
category. By definition, a left $\mathcal{C}$-module is a covariant linear
functor from $\mathcal{C}$ to $\Bbbk $-$\Mod$. It is easy to see that a $%
\mathcal{C}$-module $M$ is a family $\{{}_{x}M\}_{x\in \mathcal{C}_{0}}$ of
linear spaces, together with maps $\cdot :{}_{x}\mathcal{C}_{y}\otimes
{}_{y}M\longrightarrow {}_{x}M$ such that 
\begin{equation*}
1_{z}\cdot m=m\qquad \text{and\qquad }f\cdot (g\cdot m)=(f\circ g)\cdot m,
\end{equation*}%
for any $m\in {}_{z}M$ and for any $f\in {}_{x}\mathcal{C}_{y}$ and $g\in
{}_{y}\mathcal{C}_{z}$. A morphism of $\mathcal{C}$-modules $u:M\rightarrow
N $ is a natural transformation between the corresponding functors.
Equivalently, a morphism $u$ as above is given by a family $\{_{x}u\}_{x\in 
\mathcal{C}_{0}}$ of $\Bbbk $-linear maps $_{x}u:{}_{x}M\rightarrow {}_{x}N$
such that, for all $f\in {}_{y}\mathcal{C}_{x}$\ and $m\in {}_{x}{M}$ we have%
\begin{equation*}
_{y}u(f\cdot m)=f\cdot {}_{x}u(m).
\end{equation*}%
The category $\Mod$-$\mathcal{C}$ of right $\mathcal{C}$-modules is defined
similarly.

\item The category $\mathcal{C}$-$\Mod$-$\mathcal{C}$\ of $\mathcal{C}$%
-bimodules will play an important role in our paper. By definition, a $%
\mathcal{C}$-bimodule is a left $\mathcal{C\boxtimes C}^{op}$-module.
Therefore, a bimodule $M$\ is a family $\{_{x}M_{y}\}_{x,y\in \mathcal{C}%
_{0}}$, endowed with maps $\cdot :{}_{x}\mathcal{C}_{y}\otimes
{}_{y}M_{z}\longrightarrow {}_{x}M_{z}$ and $\cdot :{}_{x}M_{y}\otimes {}_{y}%
\mathcal{C}_{z}\longrightarrow {}_{x}M_{z}.$ By definition, for any $y\in 
\mathcal{C}_{0},$ the family $_{\bullet }M_{y}:=\{_{x}M_{y}\}_{x\in \mathcal{%
C}_{0}}$ is a left module with respect to the former maps. Similarly, for
any $x\in \mathcal{C}_{0}$ the family $_{x}M_{\bullet }:=\{_{x}M_{y}\}_{y\in 
\mathcal{C}_{0}}$ is a right module with respect to the latter maps. In
addition, these module structures are compatible in the sense that, for any $%
m\in {}_{y}{M}_{z},$ and any morphisms $f\in {}_{x}\mathcal{C}_{y}$ and $%
g\in {}_{z}\mathcal{C}_{t},$ we have%
\begin{equation*}
(f\cdot m)\cdot g=f\cdot (m\cdot g).
\end{equation*}%
A morphism of bimodules $u:M\rightarrow N$ is given by a family $%
\{_{x}u_{y}\}_{x,y\in \mathcal{C}_{0}}$ of $\Bbbk $-linear transformations $%
_{x}u_{y}:{}_{x}M_{y}\longrightarrow {}_{x}N_{y}$, which for $m,$ $f$ and $g$
as above verify the relation%
\begin{equation*}
_{x}u_{t}(f\cdot m\cdot g)=f\cdot {}_{y}u_{z}(m)\cdot g.
\end{equation*}
\end{enumerate}

It is well-known that $\mathcal{C}$-$\Mod$ has enough injective objects, for
any linear category $\mathcal{C}$. In $\mathcal{C}$-$\Mod$ there are enough
projective objects as well. Clearly, $\Mod$-$\mathcal{C}$ and $\mathcal{C}$-$%
\Mod$-$\mathcal{C}$ also have the above properties, since they can be
regarded as categories of left modules over $\mathcal{C}^{op}$ and $\mathcal{%
C\boxtimes C}^{op},$ respectively.
\end{fact}

\begin{fact}[The tensor product of two (bi)modules.]
Let $\mathcal{C}$ be a linear category. We assume that ${M}$ is a right $%
\mathcal{C}$-module and that $N$ is a left $\mathcal{C}$-module. The vector
space${\ M}\otimes _{\mathcal{C}}N$ is defined as the quotient of $%
\oplus_{z\in \mathcal{C}_{0}}(M_{z}\otimes {}_{z}N)$ through the linear
subspace generated by the elements $m\cdot f\otimes n-m\otimes f\cdot n,$
for arbitrary $m\in {}M_{u},$ $f\in {}_{u}\mathcal{C}_{v}$ and $n\in
{}_{v}N. $ The class of $m\otimes n\in {}M_{u}\otimes {}_{u}N$ in the
quotient linear space ${M}\otimes _{\mathcal{C}}N$ will be denoted by $%
m\otimes _{\mathcal{C}}n.$ These tensor monomials generate ${M}\otimes _{%
\mathcal{C}}N$ as a vector space.

Now one can define the tensor product of two bimodules $X$ and $Y$ as
follows. For $x$ and $y$ in $\mathcal{C}_{0}$ we know that $_{x}X_{\bullet }$
is a a right module and that $_{\bullet }Y_{y}$ is left module. Thus the
tensor product $_{x}{X}_{\bullet }\otimes _{\mathcal{C}}{}_{\bullet }Y_{y}$\
makes sense. By definition $X\otimes _{\mathcal{C}}Y$ is the bimodule whose
components are the vector spaces%
\begin{equation*}
_{x}({X}\otimes _{\mathcal{C}}Y)_{y}:={}{X}_{\bullet }\otimes _{\mathcal{C}%
}{}_{\bullet }Y_{y}.
\end{equation*}%
The bimodule structure on ${X}\otimes _{\mathcal{C}}Y$ is induced by the
left action on ${X}$ and the right action on $Y$.

Note that the category $\mathcal{C}$-$\Mod$-$\mathcal{C}$ is a monoidal
category with respect to the tensor product of $\mathcal{C}$-bimodules. Its
unit object is the bimodule $\mathcal{C}$. In $\mathcal{C}$-$\Mod$-$\mathcal{%
C}$ there are arbitrary coproducts, and it is easy to see that the tensor
product $\otimes _{\mathcal{C}}$ is distributive over coproducts.
\end{fact}

\begin{fact}[Exemples of $\mathcal{C}$-bimodules.]
By definition, $\mathcal{C}\otimes \mathcal{C}$ is the $\mathcal{C}$%
-bimodule whose components are the vector spaces%
\begin{equation*}
_{x}(\mathcal{C}\otimes \mathcal{C})_{y}=\textstyle\bigoplus\limits_{z\in 
\mathcal{C}_{0}}(_{x}\mathcal{C}_{z}\otimes {}_{z}\mathcal{C}_{y}).
\end{equation*}%
The left and right actions are induced by the composition in $\mathcal{C}$.
Thus the bimodule structure is defined by the relation%
\begin{equation*}
f^{\prime }\cdot (g^{\prime \prime \prime })\cdot f^{\prime \prime
}=(f^{\prime }\circ g^{\prime })\otimes (g^{\prime \prime }\circ f^{\prime
\prime }),
\end{equation*}%
where $f^{\prime }$, $f^{\prime \prime }$, $g^{\prime }$ and $g^{\prime
\prime }$ are arbitrary morphisms in $\mathcal{C}$ such that $f^{\prime
}\circ g^{\prime }$ and $g^{\prime \prime }\circ f^{\prime \prime }$ make
sense. One can prove easily that $\mathcal{C}\otimes \mathcal{C}$ is
projective as a $\mathcal{C}$-bimodule.

Another example of $\mathcal{C}$-bimodule is $\mathcal{C}$ itself. Its
components are the linear spaces $_{x}\mathcal{C}_{y}$\ and the actions are
defined by the composition in $\mathcal{C}$. Therefore, for any object $x$
in $\mathcal{C}$ we can consider the right $\mathcal{C}$-module $_{x}%
\mathcal{C}_{\bullet }.$ In view of the remarks from the previous
subsection, for any family $\mathcal{I}:=\{x_{i}\}_{i\in I}$ of objects in $%
\mathcal{C}_{0},$ we can consider the coproduct $\mathcal{C}{(}\mathcal{I}{)}%
:=\oplus_{i\in I}{}_{x_{i}}\mathcal{C}_{\bullet }$ in the category of right $%
\mathcal{C}$-modules. Of course, a similar coproduct exists in the category
of left $\mathcal{C}$-modules. 
\end{fact}

\begin{fact}[Finitely generated projective $\mathcal{C}$-modules and the
trace map.]
\label{fa: fgp}We fix, as before, a $\Bbbk $-linear category $\mathcal{C}$.
Let $M$ be a right $\mathcal{C}$-module. By definition, $M$ is finitely
generated if there is a finite family $\mathcal{I}:=\{x_{i}\}_{i\in I}$ of
objects in $\mathcal{C}$ such that $M$ is a quotient of ${}\mathcal{C}{(}%
\mathcal{I}{)}$. Equivalently, there is a set $\left\{ m_{i}\right\} _{i\in
I}$, with $m_{i}\in M_{x_{i}}$, such that any $m\in M_{x}$ can be written as
a linear combination%
\begin{equation}
m=\textstyle\sum\limits_{i\in I}m_{i}\cdot f_{i},  \label{ec:m}
\end{equation}%
for some $f_{i}\in {}_{x_{i}}\mathcal{C}_{x}$. If, in addition, $M$ is
projective then the canonical projection from $\mathcal{C}{(}\mathcal{I}{)}$
to $M$ has a section in $\Mod$-$\mathcal{C}.$ Therefore $M$ is a direct
summand of $\mathcal{C}{(\mathcal{I})}$. The converse also holds, as $%
\mathcal{C}{(}\mathcal{I}{)}$ is a finitely generated projective right $%
\mathcal{C}$-module, for any finite family $\mathcal{I}$ of objects in $%
\mathcal{C}.$

From the above characterization we deduce that $(M,\cdot )$ is a finitely
generated projective right $\mathcal{C}$-module if and only if there are 
\emph{dual bases} $\{m_{i}\}_{i\in I}$ and $\{\varphi ^{i}\}_{i\in I}$ on $M$%
. By definition, $\mathcal{I}:=\{x_{i}\}_{i\in I}$ is a finite family of
objects, $m_{i}\in M_{x_{i}}$ and $\varphi ^{i}:M\rightarrow {}_{x_{i}}%
\mathcal{C}_{\bullet }$ is a morphism of right $\mathcal{C}$-modules, for
all $i\in I$. We shall denote the components of $\varphi ^{i}$ by $\varphi
_{x}^{i}:M_{x}\rightarrow {}_{x_{i}}\mathcal{C}_{x}$. In addition, for any $%
m\in M_{x}$, the following relation holds%
\begin{equation}
m=\textstyle\sum\limits_{i\in I}m_{i}\cdot \varphi _{x}^{i}(m).
\label{ec:phi}
\end{equation}%
Using the existence of dual bases for finitely generated projective $%
\mathcal{C}$-modules, one proves that%
\begin{equation}
\Hom_{\mathcal{C}}(M,M)\overset{\cong }{\longrightarrow }M\otimes _{\mathcal{%
C}}M^{\ast },  \label{ec:izo}
\end{equation}%
where $M^{\ast }$ is the dual of $M.$ By definition, $M^{\ast }$ is a left $%
\mathcal{C}$-module, and its components are the linear spaces $_{x}M^{\ast
}:=\Hom_{\mathcal{C}}(M,{}_{x}\mathcal{C}_{\bullet })$. If $\varphi
=\{\varphi _{y}\}_{y\in \mathcal{C}_{0}}$ is a morphism of right $\mathcal{C}
$-modules from $M$ to $_{x}\mathcal{C}_{\bullet }$ and $f\in {}_{y}\mathcal{C%
}_{x},$ then the component $\left( f\cdot \varphi \right)
_{z}:M_{z}\rightarrow {}_{y}\mathcal{C}_{z}$ of $f\cdot \varphi \in
{}_{y}M^{\ast }$ maps $m\in M_{z}$ to $f\cdot \varphi _{z}(m).$

As in the case of $\Bbbk $-algebras, we can speak about the commutator $%
[f,g] $ of two morphisms $f\in {}_{x}\mathcal{C}_{y}$ and $g\in {}_{y}%
\mathcal{C}_{x}.$ By definition, $[f,g]:=f\circ g-g\circ f$ is an element in 
$\oplus _{z\in \mathcal{C}_{0}}{}_{z}\mathcal{C}_{z}.$ Let $\mathcal{C}_{ab}$
denote the quotient vector space of $\oplus _{z\in \mathcal{C}_{0}}{}_{z}%
\mathcal{C}_{z}$ through the subspace $[\mathcal{C},\mathcal{C}]$ spanned by
all commutators in $\mathcal{C}.$ The class of $f\in{}_x\mathcal{C}_x$ in $%
\mathcal{C}_{ab}$ will be denoted by $f+[\mathcal{C},\mathcal{C}]$.

If $M$ is a right $\mathcal{C}$-module, then the \textit{evaluation map} $\ev%
_{M}:M\otimes _{\mathcal{C}}M^{\ast }\rightarrow \mathcal{C}_{ab}$ is
uniquely defined by%
\begin{equation*}
\ev_{M}(m\otimes _{\mathcal{C}}\varphi ):=\varphi _{x}(m)+[\mathcal{C},%
\mathcal{C}],
\end{equation*}%
for any morphism of right $\mathcal{C}$-modules $\varphi =\{\varphi
_{y}\}_{y\in \mathcal{C}_{0}}$ in $_{x}M^{\ast }$ and $m\in M_{x}$.

We can now define the \emph{Hattori-Stallings trace map} $\Tr_{M}:\Hom_{%
\mathcal{C}}(M,M)\rightarrow \mathcal{C}_{ab},$ for any finitely generated
projective right $\mathcal{C}$-module $M,$ as being the composition of the
evaluation map $\ev_{M}$ with the isomorphism (\ref{ec:izo}). Therefore, if $%
\{m_{i}\}_{z\in I}$ and $\{\varphi ^{i}\}_{z\in I}$ are finite dual bases on 
$M,$ and $u:=\{u_{x}\}_{x\in \mathcal{C}_{0}}$ is an endomorphism of $M,$
then%
\begin{equation}
\Tr_{M}(u)=\textstyle\sum\limits_{i\in I}\varphi
_{x_{i}}^{i}(u_{x_{i}}(m_{i}))+[\mathcal{C},\mathcal{C}].  \label{ec:Tr}
\end{equation}%
Note that the formula for $\Tr_{M}(u)$ does not depend on the choice of the
dual bases. Furthermore, if $v:N\rightarrow M$ and $w:M\rightarrow N$ are
morphisms between two finitely generated projective right $\mathcal{C}$%
-modules, then by using \eqref{ec:Tr} one can easily see that $\Tr%
_{M}(v\circ w)=\Tr_{N}(w\circ v).$
\end{fact}

\section{Connections. The curvature of a connection.}

In this section we introduce the main tools that we need for the
construction of the Chern map. We start by recalling the definition of $DG$%
-categories, which explicitly appear in the definition of connections.

\begin{fact}[$DG$-categories.]
We say that a category $\Omega ^{\ast }$ is graded if every hom-space $%
_{x}\Omega _{y}^{\ast }$ is endowed with a decomposition $_{x}\Omega
_{y}^{\ast }=\oplus _{n\geq 0\,}${}$_{x}\Omega _{y}^{n}$ such that $f\circ
g\in {}_{x}\Omega _{z}^{n+m},$ for any $f\in {}_{x}\Omega _{y}^{n}$ and $%
g\in {}_{y}\Omega _{z}^{m}.$ For simplicity, the composition of two forms in
a $DG$-category $\Omega ^{\ast }$ will be denoted by concatenation, that is $%
\omega \circ \zeta =\omega \zeta .$

A differential $d^{\ast }$ on $\Omega ^{\ast }$ is a family of linear maps $%
_{x}d_{y}^{n}:{}_{x}\Omega _{y}^{n}\rightarrow {}_{x}\Omega _{y}^{n+1}$ such
that $_{x}d_{y}^{n+1}\circ {}_{x}d_{y}^{n}=0$ and%
\begin{equation}
_{x}d_{z}^{n+m}(\omega \zeta )={}_{x}d_{y}^{n}(\omega )\zeta +(-1)^{n}\omega
\,{}_{y}d_{z}^{m}(\zeta ),  \label{ec:Leibniz2}
\end{equation}%
for $\omega \in {}_{x}\Omega _{y}^{n}$ and $\zeta \in {}_{y}\Omega _{z}^{m}$%
. We refer to the relation (\ref{ec:Leibniz2}) saying that $d^{\ast }$
satisfies the (graded) Leibniz rule.

Recall that a $DG$-category is a couple $(\Omega ^{\ast },d^{\ast }),$ with $%
\Omega ^{\ast }$ a graded category and $d^{\ast }\ $a differential on $%
\Omega ^{\ast }.$ If $(\Omega ^{\ast },d^{\ast })$ is a $DG$-category, then
the vector spaces $\left\{ _{x}\Omega _{y}^{0}\right\} _{x,y\in \Omega
_{0}^{\ast }}$ define a subcategory $\Omega ^{0}$ of $\Omega ^{\ast }$.
Moreover, for every $n\in \mathbb{N},$ the family $\Omega ^{n}:=\left\{
_{x}\Omega _{y}^{n}\right\} _{x,y\in \Omega _{0}^{\ast }}$ has a natural
structure of $\Omega ^{0}$-bimodule.

An element $\omega \in {}_{x}\Omega _{y}^{n}$ will be called a differential
form in $\Omega ^{\ast }$ from $y$ to $x$ of degree $n.$ The degree of $%
\omega $ will also be denoted by $\left\vert \omega \right\vert $ and, if
there is no danger of confusion, we shall write $d\omega$ instead of $%
{}_{x}d_{y}^{n}(\omega ).$
\end{fact}

\begin{fact}[Connections.]
\label{fact:conn}Let $M$ be a finitely generated right module over a linear
category $\mathcal{C}$. We fix a $DG$-category $(\Omega ^{\ast },d^{\ast })$
such that $\Omega ^{0}=\mathcal{C}$. Since $\Omega ^{1}$ is a $\mathcal{C}$%
-bimodule, the tensor product $M\otimes _{\mathcal{C}}{}_{\bullet }\Omega
_{x}^{1}$ makes sense for any $x\in\mathcal{C}_0$.

By definition, a \emph{connection} $\nabla :M\rightarrow M\otimes _{\mathcal{%
C}}\Omega ^{1}$ on $M$ (with respect to $\Omega ^{\ast })$ is a family $%
\{\nabla _{x}\}_{x\in \mathcal{C}_{0}}$ of $\Bbbk $-linear maps $\nabla
_{x}:M_{x}\rightarrow M\otimes _{\mathcal{C}}{}_{\bullet }\Omega _{x}^{1}$
such that, for all $m\in M_{x}$ and $f\in {}_{x}\mathcal{C}_{y}$,%
\begin{equation}
\nabla _{y}(m\cdot f)=\nabla _{x}(m)\cdot f+m\otimes _{\mathcal{C}}df.
\label{ec:conn}
\end{equation}%
By assumption, there is a finite set $\left\{ m_{i}\right\} _{i\in I}$ of
generators for $M,$ with $m_{i}\in M_{x_{i}}$ for some $x_{i}\in \mathcal{C}%
_{0}.$ In view of equation \eqref{ec:conn}, the connection $\nabla $ is
uniquely determined by the elements $\lambda _{ij}\in {}_{x_{i}}\Omega
_{x_{j}}^{1}$ such that%
\begin{equation}
\nabla _{x_{i}}(m_{i})=\textstyle\sum\limits_{j\in I}m_{j}\otimes _{\mathcal{%
C}}\lambda _{ji}.  \label{ec:lambda}
\end{equation}%
Indeed, if $m\in M_{x},$ then there is $\{f_{i}\}_{i\in I}$ in $\mathcal{C}{(%
}\mathcal{I}{)}_{x}$ such that the relation (\ref{ec:m}) holds. Thus,%
\begin{equation}
\nabla _{x}(m)=\textstyle\sum\limits_{i,j\in I}m_{i}\otimes _{\mathcal{C}%
}\lambda _{ij}f_{j}+\textstyle\sum\limits_{i\in I}m_{i}\otimes _{\mathcal{C}%
}df_{i}.  \label{ec:coord}
\end{equation}%
We shall denote the matrix $(\lambda _{ij})_{i,j\in I}$ by $\Lambda $.
\end{fact}

In the following lemma we give two methods for constructing new connections.

\begin{lemma}
\label{le:conn}Let $(\Omega ^{\ast },d^{\ast })$ be a $DG$-category such
that $\Omega ^{0}=\mathcal{C}.$

\begin{enumerate}
\item[(i)] Let $\nabla _{p}:M_{p}\rightarrow M_{p}\otimes _{\mathcal{C}%
}\Omega ^{1}$ be a connection on $M_{p},$ where $p\in \{1,2\}$. There exists
a unique connection $\nabla _{1}\oplus \nabla _{2}$ on $M_{1}\oplus M_{2}$
such that the diagram below is commutative for any $p\in \{1,2\}$.%
\begin{equation*}
\xymatrix{ M_{1}\oplus M_{2}\ar[r]^-{\nabla_1\oplus\nabla_2} & \left(
M_{1}\oplus M_{2}\right) \otimes _{\mathcal{C}}\Omega ^{1} \\
M_{p}\ar[r]_-{\nabla_p}\ar@{^{(}->}[u] & M_{p} \otimes _{\mathcal{C}}\Omega
^{1}\ar@{^{(}->}[u] }
\end{equation*}

\item[(ii)] A connection $\nabla:M\to M\otimes_{\mathcal{C}}\Omega^1$
induces a connection on every direct summand $N$ of $M$.
\end{enumerate}
\end{lemma}

\begin{proof}
Let us sketch the proof of the first part of the lemma. The details are left
to the reader. Let $m_{p}$ be an element in $\left( M_{p}\right) _{x},$
where $p\in \{1,2\}$. Then%
\begin{equation*}
\nabla _{p}(m_{p})=\textstyle\sum_{j=1}^{n_{p}}m_{pj}\otimes _{\mathcal{C}%
}\omega _{pj},
\end{equation*}%
for some $m_{pj}\in \left( M_{p}\right) _{y_{j}}$ and $\omega _{pj}\in
{}_{y_{j}}\Omega _{x}^{1}.$ We now define the connection $\nabla _{1}\oplus
\nabla _{2}$ by%
\begin{equation*}
\left( \nabla _{1}\oplus \nabla _{2}\right) _{x}(m_{1},m_{2})=\textstyle%
\sum_{j=1}^{n_{1}}\left( m_{1j},0\right) \otimes _{\mathcal{C}}\omega _{1j}+%
\textstyle\sum_{j=1}^{n_{2}}\left( 0,m_{2j}\right) \otimes _{\mathcal{C}%
}\omega _{2j}.
\end{equation*}%
To prove the second part of the lemma, we choose a morphism $\pi
:M\rightarrow N$ in $\Mod$-$\mathcal{C}$ which has a section $\sigma
:N\rightarrow M$ in this category. Then, for an object $x$ in $\mathcal{C},$
we define $\nabla _{x}^{\prime }:N_{x}\rightarrow N\otimes _{\mathcal{C}%
}{}_{\bullet }\Omega _{x}^{1}$ by%
\begin{equation*}
\nabla _{x}^{\prime }=(\pi \otimes _{\mathcal{C}}\Id_{\Omega ^{1}\mathcal{C}%
})_{x}\circ \nabla _{x}\circ \sigma _{x}.
\end{equation*}%
It is easy to see that $\nabla ^{\prime }:=\{\nabla _{x}^{\prime }\}_{x\in 
\mathcal{C}_{0}}$ is a connection on $N$.
\end{proof}

\begin{fact}[Examples of connections.]
\label{fa:ExConn}We start by describing all connections on the right module $%
\mathcal{C}{(}\mathcal{I}{)}$, where $\mathcal{I}:=\{x_{i}\}_{i\in I}$ is a
finite family of objects in $\mathcal{C}.$ Clearly, $\left\{ e_{{i}}\right\}
_{i\in I}$ is a set of generators for $M$, where $e_{i}\in \mathcal{C}{(}%
\mathcal{I}{)}_{x_{i}}$ is the family $\{\delta _{i,j}1_{x_{i}}\}_{j\in I}.$
Here $\delta _{i,j}$ denotes the Kronecker symbol. Hence, for an element $%
\{f_{i}\}_{i\in I}$ in $\mathcal{C}{(}\mathcal{I}{)}_{x},$ we get%
\begin{equation}
\nabla _{x}(\{f_{i}\}_{i\in I})=\nabla _{x}(\textstyle\sum\limits_{i\in
I}e_{i}\cdot f_{i})=\textstyle\sum\limits_{i,j\in I}e_{i}\otimes _{\mathcal{C%
}}\lambda _{ij}f_{j}+\textstyle\sum\limits_{i\in I}e_{i}\otimes _{\mathcal{C}%
}df_{i},  \label{ec:conn_free}
\end{equation}%
where $\lambda _{ij}\in {}_{x_{i}}\Omega _{x_{j}}^{1}$. Conversely, if $%
\Lambda :=(\lambda _{ij})_{i,j\in I}$ is an arbitrary matrix of differential
forms of degree $1$ such that $\lambda _{ij}\in {}_{x_{i}}\Omega
_{x_{j}}^{1},$ then \eqref{ec:conn_free} defines a connection $\nabla
^{\Lambda }$ on $\mathcal{C}{(}\mathcal{I}{)}.$ In particular, if $\Lambda
=0 $, then%
\begin{equation}
\nabla _{x}(\{f_{i}\}_{i\in I})=\textstyle\sum\limits_{i\in I}e_{i}\otimes _{%
\mathcal{C}}df_{i}  \label{ec:d2}
\end{equation}%
defines a connection on $\mathcal{C}{(}\mathcal{I}{)}.$

Another important example is the Levi-Civita connection, which is defined on
a finitely generated projective module $M$. Let $\{m_{i}\}_{i\in I}$ and $%
\{\varphi ^{i}\}_{i\in I}$ be dual bases on $M.$ Recall that $m_{i}\in
M_{x_{i}}$ and $\varphi ^{i}$ is a linear map from $M$ to $_{x_{i}}\mathcal{C%
}_{\bullet }$. If $\sigma _{x}:M_{x}\rightarrow \mathcal{C}{(}\mathcal{I}{)}%
_{x}$ is the application%
\begin{equation*}
\sigma _{x}(m)=\left\{ \varphi _{x}^{i}(m)\right\} _{i\in I},
\end{equation*}%
then $\sigma =\{\sigma _{x}\}_{x\in \mathcal{C}_{0}}$ is a morphism in $\Mod$%
-$\mathcal{C}.$ Moreover, $\sigma $ is a section of the morphism $\pi :%
\mathcal{C}{(}\mathcal{I}{)}\rightarrow M$ whose component $\pi _{x}$ maps $%
\{f_{i}\}_{i\in I}\in \mathcal{C}{(}\mathcal{I}{)}_{x}$ to $\textstyle%
\sum\nolimits_{i\in I}m_{i}\cdot f_{i}.$ Thus $M$ can be regarded via $\pi $
and $\sigma $ as a direct summand of $\mathcal{C}{(}\mathcal{I}{)}.$ By the
second part of the preceding lemma, the connection on $\mathcal{C}{(}%
\mathcal{I}{)}$ given by the equation (\ref{ec:d2}) induces a connection $%
\nabla ^{LC}$ on $M$. We call $\nabla ^{LC}$ the Levi-Civita connection on $%
M.$ For every $m\in M_{x}$ we have%
\begin{equation}
\nabla _{x}^{LC}(m)=\textstyle\sum\limits_{i\in I}m_{i}\otimes _{\mathcal{C}%
}d\varphi _{x}^{i}(m).  \label{ec:LC}
\end{equation}
\end{fact}

\begin{fact}[The curvature of a connection.]
We fix a connection $\nabla :M\rightarrow M\otimes _{\mathcal{C}}\Omega ^{1}$
on a finitely generated right $\mathcal{C}$-module $M$. Our goal is to
extend $\nabla $ to an endomorphism $R(\nabla )$ of degree $2$ of the graded 
$\Omega ^{\ast }$-module $M\otimes _{\mathcal{C}}\Omega ^{\ast }.$ First,
for $n\geq 0$ and $x\in \mathcal{C}_{0},$ we define $\nabla
_{x}^{n}:(M\otimes _{\mathcal{C}}\Omega ^{n})_{x}\rightarrow (M\otimes _{%
\mathcal{C}}\Omega ^{n+1})_{x}$ by%
\begin{equation*}
\nabla _{x}^{n}(m\otimes _{\mathcal{C}}\omega ):=\nabla _{x}(m)\cdot \omega
+(-1)^{n}m\otimes _{\mathcal{C}}d\omega .
\end{equation*}%
In the above defining relation $m$ and $\omega $ are elements in $M_{y}$ and 
${}_{y}\Omega _{x}^{n},$ respectively. By equation (\ref{ec:conn}), it
easily follows that $\nabla ^{n}$ is well defined, that is 
\begin{equation*}
\nabla _{x}^{n}(m\cdot f\otimes _{\mathcal{C}}\zeta )=\nabla
_{x}^{n}(m\otimes _{\mathcal{C}}f\zeta ),
\end{equation*}%
for any $m\otimes f\otimes \zeta \in M_{y}\otimes {}_{y}\mathcal{C}%
_{z}\otimes {}_{z}\Omega _{x}^{n}$. Since $d^{\ast }$ satisfies the Leibniz
rule, we get%
\begin{equation*}
\nabla _{x}^{n+m}(m\otimes _{\mathcal{C}}\omega \zeta )=\nabla
_{z}^{n}(m\otimes _{\mathcal{C}}\omega )\cdot \zeta +(-1)^{n}m\otimes _{%
\mathcal{C}}\omega d\zeta ,
\end{equation*}%
for all $m\otimes \omega \otimes \zeta \in {}M_{y}\otimes {}_{y}\Omega
_{z}^{n}\otimes {}_{z}\Omega _{x}^{m}$. Now it is not difficult to prove
that the family $R(\nabla )=\{R(\nabla )_{x}\}_{x\in \mathcal{C}_{0}}$
defines a morphism of graded right $\Omega ^{\ast }$-modules, where 
\begin{equation*}
R^{n}(\nabla )_{x}:(M\otimes _{\mathcal{C}}\Omega ^{n})_{x}\rightarrow
(M\otimes _{\mathcal{C}}\Omega ^{n+2})_{x},\quad R^{n}(\nabla )_{x}=\nabla
_{x}^{n+1}\circ \nabla _{x}^{n}.
\end{equation*}%
We shall say that $R(\nabla )$ is the \emph{curvature} of $\nabla $.

Note that $M_{x}$ can be identified with $(M\otimes _{\mathcal{C}}\Omega
^{0})_{x}$ via the map $m\mapsto m\otimes _{\mathcal{C}}1_{x}.$ On the other
hand, if $\{m_{i}\}_{i\in I}$ is a set of generators of $M,$ such that $%
m_{i} $ belongs to a certain $M_{x_{i}},$ then $M\otimes _{\mathcal{C}%
}\Omega ^{\ast }$ is generated by $\{m_{i}\otimes _{\mathcal{C}%
}1_{x_{i}}\}_{i\in I}. $ It follows that the curvature of $\nabla $ in
completely determined by its values$\ $at each $m_{i}.$ Using the matrix $%
\Lambda =\left( \lambda _{ij}\right) _{i,j\in I}$ that corresponds to $%
\nabla $ and taking into account (\ref{ec:lambda}), we obtain%
\begin{equation*}
R(\nabla )_{x_{i}}(m_{i})=\textstyle\sum\limits_{j\in I}\nabla (m_{j})\cdot
\lambda _{ji}+\textstyle\sum\limits_{j\in I}m_{j}\otimes _{\mathcal{C}%
}d\lambda _{ji}=\textstyle\sum\limits_{j,k\in I}m_{j}\otimes _{\mathcal{C}%
}\lambda _{jk}\lambda _{ki}+\textstyle\sum\limits_{j\in I}m_{j}\otimes _{%
\mathcal{C}}d\lambda _{ji}.
\end{equation*}%
Hence $R(\nabla )$ satisfies the relation%
\begin{equation}
R(\nabla )_{x_{i}}(m_{i})=\textstyle\sum\limits_{j\in I}m_{j}\otimes _{%
\mathcal{C}}\gamma _{ji},  \label{ec:R}
\end{equation}%
where $\gamma _{ij}=d\lambda _{ij}+\textstyle\sum\nolimits_{k\in I}\lambda
_{ik}\lambda _{kj}.$ If $\Gamma =(\gamma _{ij})_{i,j\in I}$, then the
relations that define the elements of $\Gamma $ are equivalent to the matrix
equation%
\begin{equation}
\Gamma =d\Lambda +\Lambda ^{2}.  \label{ec:Gamma}
\end{equation}%
For example, let $\nabla ^{\Lambda }$ be the connection on $\mathcal{C}(%
\mathcal{I})$ that we constructed in \S \ref{fa:ExConn} for a matrix $%
\Lambda =(\lambda _{ij})_{i,j\in I}$. Since $\{e_{i}\}_{i\in I}$ generates $%
\mathcal{C}(\mathcal{I}),$ and the matrix corresponding to $\nabla ^{\Lambda
}$ is precisely $\Lambda ,$ the curvature $R(\nabla ^{\Lambda })$ verifies
the formula (\ref{ec:R}), in which one substitutes $m_{i}$ by $e_{i}$, and $%
\Gamma =(\gamma _{ij})_{i,j\in I}$ is given by (\ref{ec:Gamma}).

In a similar way one computes the curvature of the Levi-Civita connection on
a finitely generated projective $\mathcal{C}$-module $M$. If $%
\{m_{i}\}_{i\in I}$ and $\{\varphi ^{i}\}_{i\in I}$ are dual bases, then by
equation (\ref{ec:LC}) we have%
\begin{equation*}
\nabla _{x_{i}}^{LC}(m_{i})=\textstyle\sum\limits_{j\in I}m_{j}\otimes _{%
\mathcal{C}}d\varphi _{x_{i}}^{j}(m_{i}).
\end{equation*}%
Hence, for the curvature of $\nabla ^{LC},$ the elements of the matrix $%
\Lambda $ are $\lambda _{ij}:=d\varphi _{x_{i}}^{j}(m_{i}).$ Since $d\Lambda
=0$, we have $\Gamma =\Lambda ^{2}.$ Thus, the curvature of the Levi-Civita
connection verifies the equations%
\begin{equation*}
R(\nabla )_{x_{i}}(m_{i})=\textstyle\sum\limits_{j,k\in I}m_{j}\otimes _{%
\mathcal{C}}d\varphi _{x_{k}}^{j}(m_{k})d\varphi _{x_{i}}^{k}(m_{i}).
\end{equation*}
\end{fact}

\begin{fact}[The powers of $R(\protect\nabla )$.]
Let $\nabla :M\longrightarrow M\otimes _{\mathcal{C}}\Omega ^{1}$ be a
connection on $M$. We have seen that $R(\nabla )$ is an endomorphism of $%
M\otimes _{\mathcal{C}}\Omega ^{\ast }$ of degree $2$. Thus $R(\nabla )^{q},$
the $q^{\text{th}}$ power of $R(\nabla ),$ is an endomorphism of degree $2q$%
. Let $\Lambda :=\left( \lambda _{ij}\right) _{i,j\in I}$ be the matrix
associated to $\nabla $ with respect to a set of generators $\{m_{i}\}_{i\in
I}.$ By induction on $q$, it follows that the elements $R(\nabla
)_{x_{i}}^{q}(m_{i})\in (M\otimes _{\mathcal{C}}\Omega ^{2q})_{x_{i}}$ are
given by the relation%
\begin{equation}
R(\nabla )_{x_{i}}^{q}(m_{i})=\textstyle\sum\limits_{i_{1},\dots ,i_{q}\in
I}m_{i_{1}}\otimes _{\mathcal{C}}\gamma _{i_{1}i_{2}}\gamma
_{i_{2}i_{3}}\cdots \gamma _{i_{q-1}i_{q}}\gamma _{i_{q}i},
\label{ec:R_nabla}
\end{equation}%
where $\Gamma =(\gamma _{ij})_{i,j\in I}$ is defined by the formula (\ref%
{ec:Gamma}).

For the connection $\nabla ^{\Lambda }$ associated to the matrix $\Lambda
=(\lambda _{ij})_{i,j\in I}$ we have seen that $\Gamma =(\gamma
_{ij})_{i,j\in I}$ is given by $\Gamma =d\Lambda +\Lambda ^{2},$ so%
\begin{equation*}
R(\nabla ^{\Lambda })_{x_{i}}^{q}(e_{i})=\textstyle\sum\limits_{i_{1},\dots
,i_{q}\in I}e_{i_{1}}\otimes _{\mathcal{C}}\gamma _{i_{1}i_{2}}\gamma
_{i_{2}i_{3}}\cdots \gamma _{i_{q-1}i_{q}}\gamma _{i_{q}i}.
\end{equation*}%
Let $\{m_{i}\}_{i\in I}$ and $\{\varphi ^{i}\}_{i\in I}$ be dual bases on a
finitely generated projective module $M.$ For the Levi-Civita connection on $%
M,$ we get%
\begin{equation*}
R(\nabla ^{LC})_{x_{i}}^{q}(m_{i})=\textstyle\sum\limits_{i_{1},\dots
,i_{2q}\in I}m_{i_{1}}\otimes _{\mathcal{C}}d\varphi
_{x_{i_{2}}}^{i_{1}}(m_{i_{2}})\cdots d\varphi
_{x_{i_{_{2q}}}}^{i_{2q-1}}(m_{i_{2q}})d\varphi _{x_{i}}^{i_{2q}}(m_{i}).
\end{equation*}
\end{fact}

\section{de Rham cohomology and the Chern map.}

In this section, which is the main part of our paper, we define the de Rham
cohomology of a linear category $\mathcal{C}$ with coefficients in a $DG$%
-category $\Omega ^{\ast }$. Our goal is to associate to every finitely
generated projective $\mathcal{C}$-module $M$ some de Rham cohomology
classes, that will lead us to the construction of the Chern map.

\begin{fact}[de Rham cohomology $H^{\ast }(\mathcal{C},\Omega ^{\ast }).$]
We fix a $\Bbbk $-linear category $\mathcal{C}$, and we suppose that $%
(\Omega ^{\ast },d^{\ast })$ is a $DG$-category such that $\Omega ^{0}=%
\mathcal{C}.$ For $\omega \in {}_{x}\Omega _{y}^{p}$ and $\zeta \in
{}_{y}\Omega _{x}^{n-p}$ we define the graded commutator%
\begin{equation}
\lbrack \omega ,\zeta ]:=\omega \zeta -(-1)^{p(n-p)}\zeta \omega ,
\label{ec:grcomm}
\end{equation}%
which is an element in $\oplus _{x\in \Omega _{0}^{\ast }\,}{}_{x}\Omega
_{x}^{n}$. The subspace spanned by all graded commutators in $\oplus _{x\in
\Omega _{0}^{\ast }}{}_{x}\Omega _{x}^{n}$ will be denoted by $[\Omega
^{\ast },\Omega ^{\ast }]^{n}.$ Note the usual commutator $\omega \zeta
-\zeta \omega $ makes sense in $\Omega ^{\ast }$ as well. However, in a
graded linear category, we shall always use the notation $[\omega ,\zeta ]$
for the graded commutator of $\omega $ and $\zeta .$

For a $DG$-category $(\Omega ^{\ast },d^{\ast })$ as above we now define 
\begin{equation*}
\Omega _{ab}^{n}:=(\textstyle\bigoplus_{x\in \mathcal{C}_{0}}{}_{x}\Omega
_{x}^{n})/\left[ \Omega ^{\ast },\Omega ^{\ast }\right] ^{n}.
\end{equation*}%
For example, $[\Omega ^{\ast },\Omega ^{\ast }]^{0}$ coincides with the
subspace $[\mathcal{C},\mathcal{C}],$ used in \S \ref{fa: fgp} to define $%
\mathcal{C}_{ab},$ the target of the Hattori-Stallings trace map. Hence, $%
\Omega _{ab}^{0}=\mathcal{C}_{ab}.$

As a consequence of the Leibniz rule, we immediately deduce that $d^{n}$
maps commutators to commutators. Hence $d^{n}$ factorizes through a linear
map $d_{ab}^{n}:\Omega _{ab}^{n}\rightarrow \Omega _{ab}^{n+1}.$ Obviously,
the sequence%
\begin{equation*}
\xymatrix{ 0\ar[r] & \Omega^0_{ab}\ar[r]^-{d^0_{ab}} &
\Omega^1_{ab}\ar[r]^-{d^1_{ab}} & \dotsb \ar[r] &
\Omega^n_{ab}\ar[r]^-{d^n_{ab}} &\Omega^{n+1}_{ab}\ar[r] &\dotsb }
\end{equation*}%
is a cochain complex $(\Omega _{ab}^{\ast },d_{ab}^{\ast }),$ that will be
called the de Rham complex. By definition, the de Rham cohomology $%
H_{DR}^{\ast }(\mathcal{C},\Omega ^{\ast })$ of $\mathcal{C}$ with respect
to the $DG$-category $(\Omega ^{\ast },d^{\ast })$ is the cohomology of $%
(\Omega _{ab}^{\ast },d_{ab}^{\ast })$.
\end{fact}

\begin{fact}[The cocycles $\Ch^{\ast}(M,\protect\nabla)$.]
\label{fa:Chern} Our aim now is to associate certain cohomology classes in $%
H_{DR}^{\ast }(\mathcal{C},\Omega ^{\ast })$ to any finitely generated
projective $\mathcal{C}$-module $M$ which is endowed with a connection $%
\nabla :M\longrightarrow M\otimes _{\mathcal{C}}\Omega ^{1}.$

First, we remark that the subspaces $\oplus {}_{n\in \mathbb{N\,}%
}{}_{x}\Omega _{y}^{2n}$ define a linear subcategory $\Omega ^{2\ast }$ of $%
\Omega ^{\ast },$ which contains $\mathcal{C}.$ The vector space $\left(
\Omega ^{2\ast }\right) _{ab},$ that corresponds to the subcategory of even
forms, is the quotient of the coproduct of the family $\{_{x}\Omega
_{x}^{2q}\}_{(q,x)\in \mathbb{N}\times \mathcal{C}_{0}}$ through the
subspace spanned by the commutators of two even forms. On the other hand, $%
\Omega _{ab}^{2q}$ is obtained by killing all commutators of degree $2q$ in $%
\oplus _{x\in \mathcal{C}_{0}\,x}\Omega _{x}^{2q},$ including those that
corresponds to two odd forms. Hence, $\left( \Omega ^{2\ast }\right) _{ab}$
and $\oplus _{q\geq 0}\Omega _{ab}^{2q}$ are not identical, but there is a
canonical linear transformation from the former vector space to the latter
one, that respects the canonical $\mathbb{N}$-gradings, in the sense that
the equivalence class of a form $\omega \in {}_{x}\Omega _{x}^{2q}$ is
mapped to its equivalence class in $\Omega _{ab}^{2q}.$ In conclusion, for
any finitely generated projective right $\Omega ^{2\ast }$-module $N,$ we
can compose the Hattori-Stallings trace map $\Tr_{N}:\End_{\Omega ^{2\ast
}}(N)\rightarrow $ $\left( \Omega ^{2\ast }\right) _{ab}$ with the above
canonical map. We still denote the resulting map by $\Tr_{N}$. Throughout
the remaining part of this paper we shall work only with this new trace map,
whose codomain is $\oplus _{q\geq 0}\Omega _{ab}^{2q}.$

Recall that, by assumption, $M$ is a finitely generated projective $\mathcal{%
C}$-module. Thus, the right $\Omega ^{\ast }$-module $M\otimes _{\mathcal{C}%
}\Omega ^{2\ast }$ has the same properties. Therefore, it makes sense to
speak about the trace of its endomorphisms, which are elements of $\oplus
_{q\geq 0}\Omega _{ab}^{2q},\ $in view of the foregoing remarks.

We have seen that $R(\nabla )^{q}$ is an endomorphism of $M\otimes _{%
\mathcal{C}}\Omega ^{\ast }$ of degree $2q.$ Thus, it induces a homogeneous
endomorphism of $M\otimes _{\mathcal{C}}\Omega ^{2\ast }\ $of the same
degree, so we can compute its trace%
\begin{equation*}
\Ch^{q}(M,\nabla ):=\Tr_{M\otimes _{\mathcal{C}}\Omega ^{2\ast }}\left(
R(\nabla )^{q}\right) .
\end{equation*}%
We claim that $\Ch^{q}(M,\nabla )$ is a $2q$-cochain in the de Rham complex $%
(\Omega _{ab}^{\ast },d_{ab}^{\ast }).$ Let $\{m_{i}\}_{i\in I}$ and $%
\{\varphi ^{i}\}_{i\in I}$ be dual bases on $M$ such that each $m_{i}$
belongs to a certain component $M_{x_{i}}.$ Then $\{m_{i}\otimes _{\mathcal{C%
}}1_{x_{i}}\}_{i\in I}$ and $\{\varphi ^{i}\otimes _{\mathcal{C}}\Id_{\Omega
^{2\ast }}\}_{i\in I}$ are dual bases on $M\otimes _{\mathcal{C}}\Omega
^{2\ast }.$ Therefore, by the definition of the trace and the relation (\ref%
{ec:R_nabla}) we get that%
\begin{align*}
\Ch^{q}(M,\nabla )& =\textstyle\sum_{i_{0}\in I}\left( \varphi
^{i_{0}}\otimes _{\mathcal{C}}\Id_{\Omega ^{2\ast }}\right) \left( R(\nabla
)^{q}(m_{i_{0}})\right) +\left[ \Omega ^{\ast },\Omega ^{\ast }\right] ^{2q}
\\
& =\textstyle\sum\limits_{i_{0},i_{1},\dots ,i_{q}=1}\varphi
_{x_{i_{1}}}^{i_{0}}(m_{i_{1}})\gamma _{i_{1}i_{2}}\ldots \gamma
_{i_{q-1}i_{q}}\gamma _{i_{q}i_{0}}+\left[ \Omega ^{\ast },\Omega ^{\ast }%
\right] ^{2q}.
\end{align*}%
In particular, our claim has been proved, as $\Ch^{q}(M,\nabla )$ is the
class of $\omega ^{q}(M,\nabla )$ in $\Omega _{ab}^{2q}$, where%
\begin{equation}
\omega ^{q}(M,\nabla )=\textstyle\sum\limits_{i_{0},i_{1},\dots
,i_{q}=1}\varphi _{x_{i_{1}}}^{i_{0}}(m_{i_{1}})\gamma _{i_{1}i_{2}}\ldots
\gamma _{i_{q-1}i_{q}}\gamma _{i_{q}i_{0}}.  \label{ec:wq}
\end{equation}%
Note that, since the trace of an endomorphism does not depend on the choice
of the dual bases on $M$, the class of $\omega ^{q}(M,\nabla )$ in $\Omega
_{ab}^{\ast }$ is also independent of such a choice.

On the projective and finitely generated module $\mathcal{C}(\mathcal{I})$
we take the dual bases $\{e_{i}\}_{z\in I}$ and $\{\varphi ^{i}\}_{i\in I}.$
The components of the morphism $\varphi ^{i}$ are the canonical projections $%
\varphi _{x}^{i}:\,_{x_{i}}\mathcal{C}(\mathcal{I})_{x}\rightarrow \,_{x_{i}}%
\mathcal{C}_{x}$ that maps $\{f_{i}\}_{i\in I}$ to $f_{i}.$ On $\mathcal{C}(%
\mathcal{I})$ we take, as usual, the connection $\nabla ^{\Lambda }$
associated to a matrix $\Lambda =(\lambda _{ij})_{i,j\in I},$ with $\lambda
_{ij}\in \,_{x_{i}}\Omega _{x_{j}}^{1}.$ Since the matrix that corresponds
to $\nabla ^{\Lambda }$ with respect to the generating set $\{e_{i}\}_{i\in
I}$ is $\Lambda ,$ and $\varphi _{x_{i_{1}}}^{i_{0}}(e_{i_{1}})=\delta
_{i_{0},i_{1}}1_{x_{i_{1}}},$ the equation (\ref{ec:wq}) is equivalent in
this setting to%
\begin{equation}
\omega ^{q}(M,\nabla ^{\Lambda })=\textstyle\sum\limits_{i_{1},\ldots
,i_{q}\in I}\gamma _{i_{1}i_{2}}\ldots \gamma _{i_{q-1}i_{q}}\gamma
_{i_{q}i_{1}},  \label{ec:wLambda}
\end{equation}%
where $\Gamma =(\gamma _{ij})_{i,j\in I}$ is given $\Gamma =d\Lambda
+\Lambda ^{2}.$

With respect to the dual bases $\{m_{i}\}_{i\in I}$ and $\{\varphi
^{i}\}_{i\in I},$ the Levi-Civita connection $\nabla ^{LC}$ on a finitely
generated projective module $M$ has the matrix $\Lambda :=(d\varphi
_{x_{j}}^{i}(m_{j}))_{i,j\in I}.$ It follows that $\omega ^{q}(M,\nabla
^{LC})$ satisfies the following equation%
\begin{equation}
\omega ^{q}(M,\nabla ^{LC})=\textstyle\sum\limits_{i_{0},i_{1},\dots
,i_{2q}}\varphi _{x_{i_{1}}}^{i_{0}}(m_{{i_{1}}})d\varphi
_{x_{i_{2}}}^{i_{1}}(m_{{i_{2}}})\cdots d\varphi _{x_{i_{2q}}}^{i_{2q-1}}(m_{%
{i_{2q}}})d\varphi _{x_{i_{0}}}^{i_{2q}}(m_{{i_{0}}}).  \label{ec:wLC}
\end{equation}
\end{fact}

\begin{theorem}
Let $\Bbbk $ be a field of characteristic different of $2.$ If $M$ is a
finitely generated projective $\mathcal{C}$-module and $\nabla :M\rightarrow
M\otimes _{\mathcal{C}}\Omega ^{1}$ is a connection on $M,$ then $\Ch%
^{q}(M,\nabla )$ is a $2q$-cocycle in the de Rham complex $(\Omega
_{ab}^{\ast },d_{ab}^{\ast }).$
\end{theorem}

\begin{proof}
We have to prove that $d\omega ^{q}(M,\nabla )$ is a sum of graded
commutators. Since $M$ is projective and finitely generated, there is a
right $\mathcal{C}$-module $N$ such that $M\oplus N=\mathcal{C}(\mathcal{I}),
$ for a certain finite subset $I$ of $\mathcal{C}_{0}$. Since $N$ is a
direct summand of $\mathcal{C}(\mathcal{I})$ we can choose dual bases $%
\{n_{i}\}_{i\in I}$ and $\{\psi ^{i}\}_{i\in I}$ on $N.$ By Lemma \ref%
{le:conn}(i), $\nabla \oplus \nabla ^{LC}$ is a connection on $M\oplus N,$
where $\nabla ^{LC}$ denotes the Levi-Civita connection on $N.$ Obviously, 
\begin{equation*}
\omega ^{q}(M,\nabla )+\omega ^{q}\left( N,\nabla ^{LC}\right) =\omega
^{q}\left( \mathcal{C}(\mathcal{I}),\nabla \oplus \nabla ^{LC}\right) .
\end{equation*}%
In conclusion, it is enough to show that $d\omega ^{q}(N,\nabla ^{LC})$ and $%
d\omega ^{q}(\mathcal{C}(\mathcal{I}),\nabla \oplus \nabla ^{LC})$ can be
written as sums of commutators.

We consider first the case of the Levi-Civita connection on $N.$ For a
matrix $X=(\omega _{ij})_{i,j\in I},$ with $\omega _{ij}\in {}_{x_{i}}\Omega
_{x_{j}}^{n}$ we define $\Tr(X)=\textstyle\sum\nolimits_{i\in I}\omega _{ii}.
$ If $Y=(\zeta _{ij})_{i,j\in I}$ is another matrix, with $\zeta _{ij}\in
{}_{x_{i}}\Omega _{x_{j}}^{m},$ then%
\begin{equation*}
\Tr(YX)=(-1)^{nm}\Tr(XY)+[\Omega ^{\ast },\Omega ^{\ast }]^{n+m}.
\end{equation*}%
By the equation (\ref{ec:wLC}) and the definition of the trace of a matrix
we get%
\begin{equation*}
d\omega ^{q}(N,\nabla ^{LC})=\textstyle\sum\limits_{i_{0},i_{1},\dots
,i_{2q}\in I}d\psi _{x_{i_{1}}}^{i_{0}}(m_{{i_{1}}})d\psi
_{x_{i_{2}}}^{i_{1}}(m_{{i_{2}}})\cdots d\psi _{x_{i_{2q}}}^{i_{2q-1}}(m_{{%
i_{2q}}})d\psi _{x_{i_{0}}}^{i_{2q}}(m_{{i_{0}}})=\Tr(\left( d\Psi \right)
^{2q+1}),
\end{equation*}%
where $\Psi :=(\psi _{x_{j}}^{i}(n_{j}))_{i,j\in I}.$ On the other hand,
since $\{n_{i}\}_{i\in I}$ and $\{\psi ^{i}\}_{i\in I}$ are dual bases we
have $n_{i}=\textstyle\sum\nolimits_{j\in I}n_{j}\cdot \psi _{xi}^{j}(n_{i}),
$ for all $i\in I.$ Therefore, for any $k\in I$, we get%
\begin{equation*}
\psi _{i}^{k}(n_{i})=\textstyle\sum\nolimits_{j\in I}\psi
_{j}^{k}(n_{j})\circ \psi _{x_{i}}^{j}(n_{i}).
\end{equation*}%
These identities are equivalent to the matrix equation $\Psi =\Psi ^{2}.$ We
can now proceed as in the proof \cite[Theorem 1.19]{Ka}. Namely, let $\Pi
=2\Psi -1$. Hence $\Pi ^{2}=1$ and $\Pi (d\Psi )=-(d\Psi )\Pi .$ Thus,%
\begin{equation*}
(d\Psi )^{2q+1}=\Pi ^{2}(d\Psi )^{2q+1}=-\Pi (d\Psi )^{2q+1}\Pi .
\end{equation*}%
By the foregoing computations we deduce that%
\begin{equation*}
d\omega ^{q}(N,\nabla ^{LC})=\Tr\left( (d\Psi )^{2q+1}\right) =-\Tr\left(
\Pi (d\Psi )^{2q+1}\Pi \right) =-\Tr\left( (d\Psi )^{2q+1}\Pi ^{2}\right)
+[\Omega ^{\ast },\Omega ^{\ast }]^{2q+1}.
\end{equation*}%
Since $2$ is invertible in $\Bbbk $ we conclude that $d\omega ^{q}(N,\nabla
^{LC})$ is a commutator.

It remains to show that $d\omega ^{q}(\mathcal{C}(\mathcal{I}),\nabla \oplus
\nabla ^{LC})$ is a commutator as well. Since $\nabla \oplus \nabla ^{LC}$
is a connection on $\mathcal{C}(\mathcal{I})$ there exists a matrix $\Lambda
=(\lambda _{ij})_{i,j\in I},$ with $\lambda _{ij}\in {}_{x_{i}}\Omega
_{x_{j}},$ such that $\nabla \oplus \nabla ^{LC}=\nabla ^{\Lambda }.$ Let $%
\Gamma :=(\gamma _{ij})_{i,j\in I}$ be the matrix $\Gamma =d\Lambda +\Lambda
^{2}.$ By (\ref{ec:wLambda}) we have%
\begin{equation*}
\omega ^{q}\left( \mathcal{C}(\mathcal{I}),\nabla ^{\Lambda }\right) =%
\textstyle\sum\limits_{i_{1},\ldots ,i_{q}\in I}\gamma _{i_{1}i_{2}}\ldots
\gamma _{i_{q-1}i_{q}}\gamma _{i_{q}i_{1}}=\Tr(\Gamma ^{q}).
\end{equation*}%
On the other hand, by induction on $q,$ one shows that $d\left( \Gamma
^{q}\right) =\Gamma ^{q}\Lambda -\Lambda \Gamma ^{q}$. As the trace map and $%
d^{\ast }$ commute we have%
\begin{equation*}
d\omega ^{2q}(\mathcal{C}(\mathcal{I}),\nabla ^{\Lambda })=\Tr\left(
d(\Gamma ^{q})\right) =\Tr\left( \Gamma ^{q}\Lambda \right) -\Tr(\Lambda
\Gamma ^{q}).
\end{equation*}%
We conclude the proof by remarking that the elements of $\Gamma $ are all of
even degree, so $\Tr\left( \Gamma ^{q}\Lambda \right) -\Tr(\Lambda \Gamma
^{q})$ is a commutator.
\end{proof}

Our next aim is to prove that the cohomology class of $\Ch^{q}(M,\nabla )$
in $H_{DR}^{2q}(\mathcal{C},\Omega ^{\ast })$ does not depend on the
connection $\nabla $. We start by proving some preliminary results. First of
all, we shall associate to a $DG$-category $(\Omega ^{\ast },d^{\ast })$ two
new $DG$-categories $\Omega ^{\ast }[t]$ and $\widetilde{\Omega }^{\ast }.$
Recall that $\Omega ^{0}=\mathcal{C}.$

\begin{fact}[{The $DG$-category $\Omega ^{\ast }[t]$.}]
By definition, the objects of $\Omega ^{\ast }[t]$ are the elements of $%
\mathcal{C}_{0}$, and we set$\ {}_{x}\Omega ^{n}[t]_{y}:={}_{x}\Omega
_{y}^{n}\otimes \Bbbk \lbrack t]$. Therefore, a morphism $\omega $ in$\
{}_{x}\Omega ^{n}[t]_{y}$ can be uniquely written as a polynomial $\omega =%
\textstyle\sum\nolimits_{i=0}^{p}\omega _{i}t^{i}$ with coefficients in $%
{}_{x}\Omega _{y}^{n}$. The composition in $\Omega ^{\ast }[t]$ is given by
the relation%
\begin{equation*}
(\textstyle\sum\nolimits_{i=0}^{p}\omega _{i}t^{i})\circ (\textstyle%
\sum\nolimits_{j=0}^{q}\zeta _{j}t^{j})=\textstyle\sum\nolimits_{k=0}^{p+q}(%
\textstyle\sum\nolimits_{r=0}^{k}\omega _{r}\zeta _{k-r})t^{k},
\end{equation*}%
while the identity morphism of $x$ in $\Omega ^{\ast }[t]$ is $1_{x}\in
\,_{x}\mathcal{C}_{x}$, regarded as a constant polynomial in ${}_{x}\Omega
^{0}[t]_{x}.$ The differential of $\Omega ^{\ast }[t]$ satisfies, for any
polynomial in $_{x}\Omega \lbrack t]_{y}^{n},$ the following relation%
\begin{equation*}
d(\textstyle\sum\nolimits_{i=0}^{p}\omega _{i}t^{i})=\textstyle%
\sum\nolimits_{i=0}^{p}\left( d\omega _{i}\right) t^{i}.
\end{equation*}
\end{fact}

\begin{fact}[The $DG$-category $\widetilde{\Omega }^{\ast }$.]
\label{fa:Omega_tilde} The set of objects in $\widetilde{\Omega }^{\ast }$
is $\mathcal{C}_{0},$ while $_{x}\widetilde{\Omega }_{y}^{n}$ is defined by%
\begin{equation*}
_{x}\widetilde{\Omega }_{y}^{n}={}_{x}\Omega ^{n}[t]_{y}\oplus {}_{x}\Omega
^{n-1}[t]_{y}.
\end{equation*}%
Note that $_{x}\widetilde{\Omega }_{y}^{0}=$\thinspace $_{x}\Omega
^{0}[t]_{y}{}.$ It is convenient to write an element $\omega $ in $_{x}%
\widetilde{\Omega }_{y}^{n}$ as formal sum $\omega =\omega _{0}+\omega
_{1}\varepsilon ,$ where $\omega _{0}\in {}_{x}\Omega ^{n}[t]_{y}$ and $%
\omega _{1}\in {}_{x}\Omega ^{n-1}[t]_{y}.$ In $\widetilde{\Omega }^{\ast }$
the composition of morphisms is defined by the formula%
\begin{equation*}
(\omega _{0}+\omega _{1}\varepsilon )\circ (\zeta _{0}+\zeta _{1}\varepsilon
)=(\omega _{0}\zeta _{0})+(\omega _{0}\zeta _{1}+(-1)^{\left\vert \zeta
_{0}\right\vert }\omega _{1}\zeta _{0}),
\end{equation*}%
and we take the identity of $x$ in ${}_{x}\widetilde{\Omega }_{x}^{\ast }$
to be $1_{x}+0\varepsilon .$ For $\omega _{0}+\omega _{1}\varepsilon \in
{}_{x}\widetilde{\Omega }_{y}^{n}$ we set%
\begin{equation*}
\partial (\omega _{0}+\omega _{1}\varepsilon )=d\omega _{0}+\left( d\omega
_{1}+(-1)^{n}\dot{\omega}_{0}\right) \varepsilon ,
\end{equation*}%
where $\dot{\omega}_{0}$ is the derivative of $\omega _{0}$ with respect to $%
t.$ Hence, for a polynomial $\zeta =\textstyle\sum\nolimits_{i=1}^{p}\zeta
_{i}t^{i}$ with coefficients in $_{x}\Omega _{y}^{n-1}$, we have $\dot{\zeta}%
=\textstyle\sum\nolimits_{i=1}^{p}i\zeta _{i}t^{i-1}$. It is not difficult
to check that $(\widetilde{\Omega }^{\ast },\partial ^{\ast })$ is a $DG$%
-category.

In the following lemma we describe the de Rham complex associated to $(%
\widetilde{\Omega }^{\ast },\partial ^{\ast }).$ To simplify the notation,
we shall write $\left\langle \omega \right\rangle $ for the class of $\omega
\in \Omega ^{n}[t]$ in $\Omega ^{n}[t]_{ab}.$
\end{fact}

\begin{lemma}
\label{le:DeRham} The de Rham complex associated to $\widetilde{\Omega }%
^{\ast }$ is isomorphic to the complex $(\mathrm{C}^{\ast },\delta ^{\ast }) 
$, where $\mathrm{C}^{n}=\Omega ^{n}[t]_{ab}\oplus \Omega ^{n-1}[t]_{ab}$,
and 
\begin{equation*}
\delta ^{n}(\left\langle \omega _{0}\right\rangle +\left\langle \omega
_{1}\right\rangle \varepsilon )=\left\langle d\omega _{0} \right\rangle
+\left\langle d \omega _{1} +(-1)^{n}\dot{\omega}_{1}\right\rangle
\varepsilon .
\end{equation*}
\end{lemma}

\begin{proof}
It is enough to prove the following identity%
\begin{equation*}
\left[ \widetilde{\Omega }^{\ast },\widetilde{\Omega }^{\ast }\right] ^{n}=%
\left[ \Omega ^{\ast }[t],\Omega ^{\ast }[t]\right] ^{n}\textstyle\bigoplus %
\left[ \Omega ^{\ast }[t],\Omega ^{\ast }[t]\right] ^{n-1}\varepsilon .
\end{equation*}%
The inclusion $\subseteq $ is a simple consequence of the following relation%
\begin{equation*}
\lbrack \omega _{0}+\omega _{1}\varepsilon ,\theta _{0}+\theta
_{1}\varepsilon ]=[\omega _{0},\theta _{0}]+\left( [\omega _{0},\theta
_{1}]+(-1)^{\left\vert \theta _{0}\right\vert }[\omega _{1},\theta
_{0}]\right) \varepsilon .
\end{equation*}%
To prove the other inclusion we notice that the following relations hold%
\begin{equation*}
\lbrack \omega ,\theta ]=[\omega +0\varepsilon ,\theta +0\varepsilon ]\quad 
\text{and\quad }[\zeta ,\xi ]\varepsilon =[\zeta +0\varepsilon ,0+\xi
\varepsilon ],
\end{equation*}%
for any commutators $[\omega ,\theta ]\in \Omega ^{n}[t]$ and $[\zeta ,\xi
]\in \Omega ^{n-1}[t].$ Therefore, $[\omega ,\theta ]+[\zeta ,\xi
]\varepsilon $ is a sum of commutators in $\widetilde{\Omega }^{n}$.
\end{proof}

\begin{fact}[The evaluation map.]
\label{fa:ev_map}Let $\Omega ^{\ast }$ be a $DG$-category. For every
polynomial $\omega _{0}=\textstyle\sum\nolimits_{i=0}^{p}\omega _{i}t^{i}$
with coefficients in $_{x}\Omega _{y}^{n}$ and $a\in \Bbbk ,$ let $\omega
_{0}(a):=\textstyle\sum\nolimits_{i=0}^{p}\omega _{i}a^{i}.$ Furthermore,
for a couple of objects $x$ and $y$ in $\mathcal{C}_{0},$ we define the
linear map $_{x}(ev_{a}^{n})_{y}$ from ${}_{x}\Omega ^{n}[t]_{y}\oplus
\Omega ^{n-1}[t]_{ab}$ to ${}_{x}\Omega _{y}^{n},$ by%
\begin{equation*}
\quad _{x}(ev_{a}^{n})_{y}(\omega _{0}+\left\langle \omega _{1}\right\rangle
\varepsilon )=\omega _{0}(a).
\end{equation*}%
Since $_{x}(ev_{a}^{n})_{x}$ maps a commutator $[\omega _{0},\theta _{0}]$
in ${}_{x}\Omega ^{n}[t]_{x}$ to the commutator $[\omega _{0}(a),\theta
_{0}(a)]$ in ${}_{x}\Omega _{x}^{n}$, the family $\{_{x}(ev_{a}^{n})_{x}%
\}_{x\in \mathcal{C}_{0}}$ induces a linear transformation%
\begin{equation*}
ev_{a}^{n}:\Omega ^{n}[t]_{ab}\oplus \Omega ^{n-1}[t]_{ab}\longrightarrow
\Omega _{ab}^{n},\quad ev_{a}^{n}\left( \left\langle \omega
_{0}\right\rangle +\left\langle \omega _{1}\right\rangle \varepsilon \right)
=\left\langle \omega _{0}(a)\right\rangle .
\end{equation*}
\end{fact}

\begin{lemma}
The family $\{ev_{a}^{n}\}_{n\in 
\mathbb{N}
}$ is a morphism of complexes between $(\widetilde{\Omega }_{ab}^{\ast
},\partial_{ab}^{\ast })$ and $(\Omega _{ab}^{\ast },d_{ab}^{\ast }).$

\begin{proof}
We identify $(\widetilde{\Omega }_{ab}^{\ast },\partial _{ab}^{\ast })$ with
the cochain complex from Lemma \ref{le:DeRham}. Therefore, if $\left\langle
\omega _{0}\right\rangle +\left\langle \omega _{1}\right\rangle \varepsilon $
is a cochain of degree $n,$ then%
\begin{equation*}
\left( ev_{a}^{n+1}\circ \partial ^{n}\right) (\left\langle \omega
_{0}\right\rangle +\left\langle \omega _{1}\right\rangle \varepsilon
)=ev_{a}^{n+1}\left( \left\langle d\omega _{0}\right\rangle +\left\langle
d\omega _{1}+(-1)^{n}\dot{\omega}_{0}\right\rangle \varepsilon \right)
=(d\omega _{0})(a).
\end{equation*}%
We conclude the proof by remarking that $\left( d\omega _{0}\right)
(a)=d(\omega _{0}(a))=\left( d^{n}\circ ev_{a}^{n}\right) (\left\langle
\omega _{0}\right\rangle +\left\langle \omega _{1}\right\rangle \varepsilon
) $.
\end{proof}
\end{lemma}

In order to prove that $\{ev_{0}^{n}\}_{n\in 
\mathbb{N}
}$ and $\{ev_{1}^{n}\}_{n\in 
\mathbb{N}
}$ induce the same maps in cohomology, we are going to construct a homotopy
map between them.

\begin{fact}[The homotopy operator.]
We keep the notation and the assumptions from \S \ref{fa:ev_map}. In
addition, we assume that the characteristic of $\Bbbk $ is zero. For any $%
\omega :=\textstyle\sum\nolimits_{i=0}^{p}\omega _{i}t^{i}$ with
coefficients in ${}_{x}\Omega _{y}^{n-1}$ we define%
\begin{equation*}
\textstyle\int\nolimits_{0}^{1}\omega dt:=\textstyle\sum%
\nolimits_{i=0}^{p}(i+1)^{-1}\omega _{i}.
\end{equation*}%
It is easy to see that the map $k^{n}:\Omega ^{n}[t]_{ab}\oplus {}\Omega
^{n-1}[t]_{ab}\varepsilon \longrightarrow {}\Omega _{ab}^{n-1}$ given by%
\begin{equation*}
\quad _{x}k_{y}^{n}(\left\langle \omega \right\rangle _{0}+\left\langle
\omega _{1}\right\rangle \varepsilon )=(-1)^{n}\langle \textstyle%
\int\nolimits_{0}^{1}\omega _{1}dt\rangle
\end{equation*}%
is well defined, as $\textstyle\int\nolimits_{0}^{1}[\zeta ,\xi ]dt$ is a
sum of commutators in $\Omega ^{\ast },$ for any $\zeta \in \,_{x}\Omega
^{p}[t]_{y} $ and $\xi \in \,_{y}\Omega ^{n-p-1}[t]_{x}.$
\end{fact}

\begin{lemma}
The operators $(-1)^{\ast }k^{\ast }$ define a homotopy between $%
ev_{1}^{\ast }$ and $ev_{0}^{\ast }$.
\end{lemma}

\begin{proof}
Let $\varpi :=\langle \omega _{0}\rangle +\langle \omega _{1}\rangle
\varepsilon $ be a cochain of degree $n$ in $(\widetilde{\Omega }_{ab}^{\ast
},\partial _{ab}^{\ast }).$ Thus%
\begin{align*}
\left( k^{n+1}\circ \partial ^{n}\right) (\varpi )& =(-1)^{n+1}\langle %
\textstyle\int\nolimits_{0}^{1}(d^{n-1}\omega _{1})dt\rangle +\langle %
\textstyle\int\nolimits_{0}^{1}\dot{\omega}_{0}dt\rangle \\
& =(-1)^{n+1}\langle d^{n-1}(\textstyle\int\nolimits_{0}^{1}\omega
_{1}dt)\rangle +\langle \omega _{0}(1)\rangle -\langle \omega _{0}(0)\rangle
\\
& =-\left( d^{n-1}\circ k^{n}\right) (\varpi )+\left(
ev_{1}^{n}-ev_{0}^{n}\right) (\varpi ).
\end{align*}%
This computation shows us that $(-1)^{\ast }k^{\ast }$ is a homotopy between 
$ev_{1}^{\ast }$ and $ev_{0}^{\ast }$.
\end{proof}

\begin{corollary}
\label{co:cobord} If $\varpi $ is a cocycle of degree $n$ in $(\widetilde{%
\Omega }_{ab}^{\ast },\partial _{ab}^{\ast })$ then $ev_{1}^{n}(\varpi
)-ev_{0}^{n}(\varpi )$ is a coboundary.
\end{corollary}

Now we can prove that the cohomology class of $\Ch^{2q}(M,\nabla)$ does not
depend on $\nabla.$ Let $B^{n}(\mathcal{C},\Omega^{\ast })$ denote the space
of $n$-coboundaries in the de Rham complex. The cohomology class of an $n$%
-cocycle $\omega$ will be denoted by $\omega+B^{n}(\mathcal{C},\Omega^{\ast
})$.

\begin{theorem}
Let $\mathcal{C}$ be a $\Bbbk $-linear category. We assume that $\Bbbk $ is
a field of characteristic zero and that $\Omega ^{\ast }$ is a $DG$-category
such that $\Omega ^{0}=\mathcal{C}.$ If $M$ is a finitely generated
projective right $\mathcal{C}$-module and $\nabla _{1},\nabla _{2}:M\to
M\otimes _{\Omega ^{0}}\Omega ^{1}$ are connections on $M$, then%
\begin{equation*}
\Ch^{q}(M,\nabla _{1})+B^{2q}(\mathcal{C},\Omega ^{\ast })=\Ch^{q}(M,\nabla
_{2})+B^{2q}(\mathcal{C},\Omega ^{\ast }).
\end{equation*}
\end{theorem}

\begin{proof}
Let $N$ be a right $\mathcal{C}$-module such that $M\oplus N=\mathcal{C}(%
\mathcal{I})$. The Levi-Civita connection $\nabla ^{LC}$ on $N$ exists, as $%
N $ is projective and finitely generated, being a direct summand of $%
\mathcal{C}(\mathcal{I})$. For $i\in \{1,2\}$ we have 
\begin{equation*}
\Ch^{q}(M\oplus N,\nabla _{i}\oplus \nabla ^{LC})=\Ch^{q}(M,\nabla _{i})+\Ch%
^{q}(N,\nabla ^{LC}).
\end{equation*}%
Therefore, by subtracting these two equations, we get%
\begin{equation*}
\Ch^{q}(M,\nabla _{1})=\Ch^{q}(M,\nabla _{2})+\Ch^{q}(M\oplus N,\nabla
_{1}\oplus \nabla ^{LC})-\Ch^{q}(M\oplus N,\nabla _{2}\oplus \nabla ^{LC}).
\end{equation*}%
In conclusion, it is enough to prove that $\Ch^{q}(\mathcal{C}(\mathcal{I}%
),\nabla )$ is a coboundary for any connection $\nabla $ on $\mathcal{C}(%
\mathcal{I}).$ We may assume that $\nabla =\nabla ^{\Lambda },$ where $%
\Lambda =(\lambda _{ij})_{i,j\in I}$ is a matrix with $\lambda _{ij}\in
{}_{x_{i}}\Omega _{x_{j}}^{1}$.

Let $\widetilde{\Omega }^{\ast }$ be the $DG$-category that we constructed
in the subsection \ref{fa:Omega_tilde}. Recall that, by construction, $%
\widetilde{\Omega }^{0}$ is the $\Bbbk $-linear category $\mathcal{C}[t]$.
It has the same objects as $\mathcal{C}$, and its morphism from $y$ to $x$
are polynomials with coefficients in $_{x}\mathcal{C}_{y}.$ On the right $%
\mathcal{C}[t]$-module $\widetilde{M}:=\mathcal{C}[t](\mathcal{I})$, we
consider the connection $\widetilde{\nabla }:\widetilde{M}\to \widetilde{M}%
\otimes _{\mathcal{C}[t]}\widetilde{\Omega }^{1}$ associated to the matrix $%
\widetilde{\Lambda }:=(\lambda _{ij}t+0\varepsilon )_{i,j\in I}=\Lambda
t+0\varepsilon .$ Note that $\lambda _{ij}t+0\varepsilon \in {}_{x_{i}}%
\widetilde{\Omega }_{x_{j}}^{1}.$

We have already proved that $\varpi :=\Ch^{q}(\widetilde{M},\widetilde{%
\nabla })$ is a $2q$-cocycle in $\widetilde{\Omega }_{ab}^{\ast }$, so by
Corollary \ref{co:cobord} it follows that $ev_{1}(\varpi )-ev_{0}(\varpi )$
is a coboundary in $(\Omega _{ab}^{\ast },d_{ab}^{\ast })$. On the other
hand, by the computation that we performed in \S \ref{fa:Chern}, the cocycle 
$\varpi $ equals the class in $\widetilde{\Omega }_{ab}^{2q}$ of the trace
of $\widetilde{\Gamma }=\partial \widetilde{\Lambda }+\widetilde{\Lambda }%
^{2}$. Since $\partial \widetilde{\Lambda }=(d\Lambda )t+\Lambda \varepsilon 
$ and $\widetilde{\Lambda }^{2}=t^{2}\Lambda ^{2}+0\varepsilon ,$ we have%
\begin{equation*}
\widetilde{\Gamma }^{q}=(\partial \widetilde{\Lambda }+\widetilde{\Lambda }%
^{2})^{q}=\left( (td\Lambda +t^{2}\Lambda ^{2})+\Lambda \varepsilon \right)
^{q}=(td\Lambda +t^{2}\Lambda ^{2})^{q}+\Lambda ^{\prime }\varepsilon ,
\end{equation*}%
where $\Lambda ^{\prime }:=\left( \lambda _{ij}^{\prime }\right) _{i,j\in I}$
is a certain matrix with $\lambda _{ij}^{\prime }\in {}_{x_{i}}\Omega
_{x_{j}}^{2q-1}.$ Hence $\varpi =\Tr\left( (td\Lambda +t^{2}\Lambda
^{2})^{q}\right) +\Tr(\Lambda ^{\prime })\varepsilon .$ By the definition of
the evaluation map, $ev_{a}(\varpi )=\Tr\left( ad^{1}\Lambda +a^{2}\Lambda
^{2})^{q}\right) .$ Thus, 
\begin{equation*}
\Ch^{q}(\mathcal{C}(\mathcal{I}),\nabla )=\Tr\left( (d^{1}\Lambda +\Lambda
^{2})^{q}\right) =ev_{1}\left( \varpi \right) -ev_{0}\left( \varpi \right) .
\end{equation*}%
We conclude that $\Ch^{q}(\mathcal{C}\left( \mathcal{I}\right) ,\nabla )$ is
a coboundary in $\Omega _{ab}^{2q}$, so the theorem is proved.
\end{proof}

\begin{fact}[The Chern classes.]
Let $M$ be a finitely generated projective $\mathcal{C}$-module, where $%
\mathcal{C}$ is a linear category over a field of characteristic zero. We
assume that $\Omega^{\ast}$ is a $DG$-category such that its homogeneous
component of degree zero equals $\mathcal{C}$. On $M$ we consider a
connection $\nabla$ on $M.$ Note that such a connection always exists, as $M$
is finitely generated and projective, so on $M$ we can take for instance the
Levi-Civita connection. We have just proved that the de Rham cohomology
class of $\Ch^{q}(M,\nabla )$ in $H^{2q}_{DR}(\mathcal{C},\Omega^{\ast})$
does not depend on $\nabla.$ We shall call this cohomology class the $q^{%
\text{th}}$ \emph{Chern class} of $M$, and we shall denote it by $\Ch%
^{q}(M,\Omega^{\ast})$.
\end{fact}

\begin{fact}[The Grothendieck group of a linear category.]
We keep the notation and the assumptions from the previous subsection. Let
us denote the isomorphism class of $M$ by $\left[ M\right] .$ The
Grothendieck group of $\mathcal{C}$ is, by definition, the quotient of the
free abelian group generated by the set $\{\left[ M\right] :M$ is finitely
generated projective$\}$ through the subgroup generated by the elements $%
\left[ M^{\prime }\right] +\left[ M^{\prime \prime }\right] -\left[
M^{\prime }\oplus M^{\prime \prime }\right] ,$ where $M^{\prime }$ and $%
M^{\prime \prime }$ are arbitrary finitely generated projective modules. We
shall denote the Grothendieck group of $\mathcal{C}$ by $K_{0}(\mathcal{C}).$

We can now prove the main result of this paper.
\end{fact}

\begin{theorem}
Let $\mathcal{C}$ be a $\Bbbk $-linear category. We assume that $\Bbbk $ is
a field of characteristic zero and that $\Omega ^{\ast }$ is a $DG$-category
such that $\Omega ^{0}=\mathcal{C}.$ The mapping $\left[ M\right] \mapsto \Ch%
^{2q}(M)$ induces a morphism of groups from $K_{0}(\mathcal{C})$ to $%
H_{DR}^{\ast }(\mathcal{C},\Omega ^{\ast }).$
\end{theorem}

\begin{proof}
Let $M^{\prime }$ and $M^{\prime \prime }$ be two finitely generated
projective $\mathcal{C}$-module. Let $\nabla ^{\prime }$ and $\nabla
^{\prime \prime }$ be connections on $M^{\prime }$ and $M^{\prime \prime },$
respectively. We have seen that there is a unique connection $\nabla
^{\prime }\oplus \nabla ^{\prime \prime }$ on $M^{\prime }\oplus M^{\prime
\prime },$ such that its restrictions to $M^{\prime }$ and $M^{\prime \prime
}$ coincide with $\nabla ^{\prime }$ and $\nabla ^{\prime \prime },$
respectively. Since the definition of the Chern class $\Ch^{q}(M^{\prime
}\oplus M^{\prime \prime })$ does not depend on the connection, and 
\begin{equation*}
\Ch^{q}(M^{\prime }\oplus M^{\prime \prime },\nabla^{\prime }\oplus
\nabla^{\prime \prime }=\Ch^{q}(M^{\prime },\nabla^{\prime })+\Ch%
^{q}(M^{\prime \prime },\nabla^{\prime \prime }),
\end{equation*}%
it follows that $\Ch^{q}(M^{\prime }\oplus M^{\prime \prime })=\Ch%
^{q}(M^{\prime })+\Ch^{q}(M^{\prime \prime }).$ Hence the theorem is proved.
\end{proof}

\noindent\textbf{Acknowledgement.} The authors of the paper were financially
supported by UEFISCDI, Contract 560/2009 (CNCSIS code ID\_69).


\begin{thebibliography}{Ka}
\bibitem[Ka]{Ka} M. Karoubi, \emph{Homologie cyclique et }$K$\emph{-th\'{e}%
orie}, Ast\'{e}risque \textbf{149 }(1987).

\bibitem[Co]{Co} A. Connes, \emph{Non-commutative differential geometry},
Publ. Math. IHES \textbf{62} (1985), 257-360.

\bibitem[Lo]{Lo} J.-L. Loday, \emph{Cyclic Homology,} second edition, A
series in Comprehensive Studies in Mathematics, vol. 301, Springer-Verlag,
Berlin-Heidelberg-New York, 2010.

\bibitem[Mi]{Mi} B. Mitchell, \emph{Rings with several objects}, Adv. Math. 
\textbf{8} (1972), 1-161.

\bibitem[We]{We} C.~Weibel, \emph{An introduction to homological algebra},
Cambridge University Press, Cambridge, 1997.
\end{thebibliography}
\end{document}